\documentclass[12pt,reqno,a4paper]{amsart}
\usepackage{amsmath,amsfonts,amsthm,amssymb,amsxtra,enumerate,mathtools,mathabx}
\usepackage[a4paper, top=3cm, bottom=3cm, left=2.1cm, right=2.1cm]{geometry}
\usepackage[normalem]{ulem}
\usepackage[utf8]{inputenc}
\usepackage[T1]{fontenc}
\usepackage{lmodern}
\usepackage{bbm}
\usepackage{mathrsfs}
\usepackage{enumerate}
\usepackage{comment}
\usepackage{placeins}
\usepackage{pgfplots}
\usepackage{mathabx}
\pgfplotsset{compat=newest}
\usepgfplotslibrary{fillbetween}
\usepackage{tikz}
\usepackage{csquotes}
\usepackage{float}

\newtheorem{theorem}{Theorem}

\newtheorem{lemma}[theorem]{Lemma}
\newtheorem{corollary}[theorem]{Corollary}
\theoremstyle{definition}

\theoremstyle{remark}

\numberwithin{equation}{section}

\newcommand{\A}{\mathbf{A}}

\newcommand{\dd}{\, \mathrm{d}}
\newcommand{\eps}{\varepsilon}
\renewcommand{\epsilon}{\varepsilon}

\newcommand{\N}{\mathbb{N}}
\newcommand{\norm}[2][]{{\left\|#2\right\|}} 
\renewcommand{\phi}{\varphi}
\newcommand{\R}{\mathbb{R}}

\renewcommand{\epsilon}{\varepsilon}

\DeclareMathOperator{\Tr}{Tr}

\newcommand{\abs}[2][]{{\left\vert#2\right\vert}}

\makeatletter
\newcommand*\bigcdot{\mathpalette\bigcdot@{.5}}
\newcommand*\bigcdot@[2]{\mathbin{\vcenter{\hbox{\scalebox{#2}{$\m@th#1\bullet$}}}}}
\makeatother

\DeclareMathOperator{\PR}{B}

\usepackage{hyperref}
\begin{document}
	
	\title[Concentration of eigenfunctions on singular Riemannian manifolds
	]{Concentration of eigenfunctions on singular Riemannian manifolds}

	\author[Charlotte Dietze]{Charlotte Dietze}
	\author[Larry Read]{Larry Read}
	\address[Charlotte Dietze]{Mathematisches Institut, Ludwig-Maximilans Universit\"at M\"unchen, Theresienstr. 39, 80333 M\"unchen, Germany;\newline
 Institut des Hautes Etudes Scientifiques, 35 Route de Chartres, Bures-sur-Yvette, France}
	\email{dietze@math.lmu.de}
	\address[Larry Read]{Mathematisches Institut, Ludwig-Maximilans Universit\"at M\"unchen, Theresienstr. 39, 80333 M\"unchen, Germany}
	\email{read@math.lmu.de}

\begin{abstract}
		We consider a compact Riemannian manifold with boundary and a metric that is singular at the boundary. The associated Laplace--Beltrami operator is of the form of a Grushin operator plus a singular potential. In a supercritical parameter regime, we identify the rate of concentration and profile of the high--frequency eigenfunctions that accumulate at the boundary. We give an application to acoustic modes on gas planets.
	\end{abstract}
 
\maketitle

\section{Introduction}
 The study of singular Riemannian spaces has attracted great attention so far, see for example \cite{Che}. In the series of works \cite{CHT1, CHT2, de_verdiere_sub_2022} Colin de Verdi\`ere, Hillairet and Tr\'elat considered small-time asymptotics of sub--Riemannian heat kernels and derived a Weyl law for a large class of models. 

  \medskip

 Let $X$ be an $n+1$ dimensional, smooth and compact manifold with boundary. We can identify the part of $X$ near boundary with $[0,1)\times M$, where $M$ is an $n$--dimensional Riemannian manifold and $\{0\}\times M$ is identified with $\partial X$. We consider on $X$ a singular Riemannian metric $g$ such that near $\partial X$,
\begin{equation}\label{eq:Grushin_metric}
     g=dx^2+x^{-\beta}g_1(x), 
 \end{equation}
 where $\beta>0$ and $x$ is a transverse coordinate to the boundary  and $g_1(x)$ is a family of smooth Riemannian metrics on $M$ depending continuously on the variable $x$. 

\medskip

We are interested in the corresponding Laplace--Beltrami operator, which locally near the boundary reads
	\begin{equation*}
		-\partial_x^2+\frac{C_\beta}{x^2}+x^\beta \Delta_{M}\ \text{with } C_\beta=\frac{\beta n}{4}\left(1+\frac{\beta n}{4}\right) ,
	\end{equation*}
see \eqref{eq:lbop} below. Note that for $\beta=2$ this corresponds to the Grushin operator
\begin{equation*}
    X^2+\sum_{i=1}^n Y_i^2
\end{equation*}
generated by the vector fields $X=\partial_x$, $Y_i=x\partial_{y_i}$ for $i=1,\ldots, n$ plus a potential $\frac{C_\beta}{x^2}$, see \cite[p.~1746]{BL}.

 \medskip

In \cite{de_verdiere_weyl_2024} the authors studied the counting function for the Laplace--Beltrami operator $\Delta_X$ on $(X,g)$, given by  
	\begin{align*}
		N(\lambda)=\#\{k\in \N\colon \lambda_k(\Delta_X)\leq\lambda\}.
	\end{align*}
 Here we consider the Laplace--Beltrami operator on $(X,g)$ as a Friedrichs extension of the quadratic form defined on smooth compactly supported functions. They found Weyl asymptotics in the limit $\lambda\rightarrow \infty$ for the supercritical case $\beta>2/n$, in the form of
	\begin{align}
		N(\lambda)&=C_{\alpha,n}v_{G}(M)\lambda^{d/2}(1+o(1)),\label{eqn:VDdHTresult}
	\end{align}
 where $d\coloneqq n\left(1+\beta/2\right)$ is the Hausdorff dimension of $(X,g)$, $G=g_1(0)$ and $C_{a,n}$ is a positive constant depending only on $\alpha$ and $n$. Bounds were previously found under slightly different assumptions in \cite{chitour}. 
 
 \medskip

Both \cite{de_verdiere_weyl_2024} and \cite{chitour} establish asymptotics for the critical case $\beta=2/n$ and the subcritical case $\beta<2/n$. Earlier, \cite{Boscain} examined the critical case with $n=1$, $\beta=2$ in an explicit model, deriving Weyl asymptotics through direct computation. Notably, the leading-order Weyl asymptotics for $\beta<2/n$ match those for a smooth non-degenerate Riemannian metric near the boundary, while the critical case $\beta=2/n$ includes an additional logarithmic factor.

    \medskip 
    
 	In \cite{de_verdiere_weyl_2024}, the authors also show that the high--frequency modes concentrate near the boundary of $X$. More precisely, define for $q\in X$
	\begin{align}\label{eqn:spectralfnc}
		\rho_\lambda(q)\coloneqq \frac{1}{N(\lambda)}\sum_{k=1}^{N(\lambda)}\abs{\phi_k(q)}^2,
	\end{align}
	where $\phi_k$ are the eigenfunctions of $\Delta_X$ that are normalised in $L^2$ with respect to the corresponding Riemannian volume measure $dv_g$. It was confirmed in \cite{de_verdiere_weyl_2024} that for $\beta\ge2/n$ this function converges to a delta at the boundary in the distributional sense, namely
	\begin{align*} 
		\lim_{\lambda\to\infty} \int_X \rho_\lambda f\dd v_g=\int_X f\frac{\delta_{\{u=0\}}\otimes \dd v_G}{v_G(M)},
	\end{align*}
    holds for any $f\in C(X)$. 

    \begin{theorem}\label{thm:main_grushin}
		Let $\rho_\lambda$ be given by \eqref{eqn:spectralfnc}, $\beta>2/n$, then there exists a function $\PR\in L^1(0,\infty)$ depending on $\beta$ with $\norm{\PR}_1=1$ that is defined in \eqref{eq:Bdef} below such that
		\begin{align*}
			\lim_{\lambda\rightarrow\infty} \int_{[0,1)\times M}\rho_\lambda(x,y) V(\sqrt{\lambda}x,y)\dd v_{g}=\int_{[0,\infty)\times M} \PR (x)V(x,y)\dd x\frac{\dd v_G(y)}{v_G(M)}
		\end{align*}
		holds for any bounded $V\in C([0,\infty)\times M)$.
	\end{theorem}
 
	Intrinsic to the result is the fact that the high--frequency states are completely localised to the boundary in the limit. Namely, for any $\epsilon>0$, there exists $L>0$ large enough such that 
    \begin{align*}
\liminf_{\lambda\rightarrow\infty}\int_{[0,L\lambda^{-1/2}]\times M}\rho_\lambda \dd v_g\ge1-\epsilon. 
    \end{align*}   
    The theorem implies that the leading-order rate of concentration towards the boundary is $\lambda^{-1/2}$, whilst the states remain uniformly distributed in $M$. 

\medskip

The paper is structured as follows: in Section \ref{sec:traceasympt} we derive asymptotics for the trace of the operator with a rescaled potential. In Section \ref{sec:pfthm1} we then use a rescaled Hellmann--Feynman type argument together with the result from Section \ref{sec:traceasympt} to prove Theorem \ref{thm:main_grushin}. Finally in Section \ref{sec:application} we show how the result is related to the concentration of acoustic modes on gas planets.
	
\section{Trace asymptotics}\label{sec:traceasympt}

 We use the notation $X_{\tilde\varepsilon}=[0,\tilde\varepsilon)\times M$. In order to find the convergence of the density, we follow an argument of adding a rescaled potential to our operator. Here and in the following, we take $V_\lambda$ as a local rescaling of $V\in C([0,\infty)\times M)$ given by
    \begin{equation}\label{eqn:Vlambda}
        V_\lambda(q)=\begin{cases}
            \lambda V(\sqrt{\lambda}x,y), &q=(x,y)\in X_1\\ 
            0, &q\in X\backslash X_1.
        \end{cases}
    \end{equation}
    In this section, we are concerned with the value of
    \begin{align}
		\label{eqn:liminfsup}\lim_{\lambda\rightarrow \infty}\begin{pmatrix}
			\inf\\ \sup
		\end{pmatrix}\lambda^{-1-\frac{d}{2}}\Tr(\Delta_X+V_\lambda-\lambda)_{-},
	\end{align}
	where we use $a_{-}=(\abs{a}-a)/2$. 

\medskip

 	We start by reducing the trace  \eqref{eqn:liminfsup} to just that of the operator on $X_\varepsilon$, for $\varepsilon<<1$ where $\Delta_{g_\varepsilon}$ has a nice form. Since the family of metrics $g_1(x)$ is continuous in $x$, for any $\delta>0$ we can find $\varepsilon>0$ and a smooth metric $g_\varepsilon$ on $X$ such that
  \begin{equation*}
     g_\varepsilon=dx^2+x^{-\beta}g_1(x=0) \ \text{on} \ X_{3\varepsilon}
 \end{equation*}
and
	\begin{equation}\label{eq:gcomp}
	    (1+\delta)^{-1}g\leq g_{\varepsilon}\leq (1+\delta)g
	\end{equation}
	holds on all of $X$.
By \eqref{eq:gcomp}, we have that 
	\begin{align*}
		(1+c_\delta)^{-1}\Delta_g\leq \Delta_{g_\varepsilon}\leq (1+c_\delta)\Delta_{g}
	\end{align*}
	in the quadratic form sense, for some $c_\delta>0$ with $c_{\delta}\rightarrow 0$ as $\delta\rightarrow 0$. Applying this inequality in estimating \eqref{eqn:liminfsup} we may without loss of generality assume that $c_\delta=0$.  

 \medskip
	
Next, we carry out a change of functions as in \cite{de_verdiere_weyl_2024}, namely, we replace any function $f\in C_c^\infty (X)$ by a smooth function that is $x^{\beta n/4} f$ on $ X_{2\varepsilon}$ and agrees with $f$ on $X\setminus  X_{3\varepsilon}$. We also perform a change of measure to make this transformation unitary and note that on $X_{2\varepsilon}$ the Riemannian volume measure $dv_g$ is now given by the Lebesgue measure on $[0,2\epsilon]$ times the corresponding Riemannian volume $dv_G$ measure on $M$, where $G=g_1(x=0)$, so in local coordinates on $X_{2\varepsilon}$, we have $dv_g=dxdv_G$.
On $X_{2\varepsilon}$ the operator $\Delta_{g_\varepsilon}$ can be written as
	\begin{equation}\label{eq:lbop}
		-\partial_x^2+\frac{C_\beta}{x^2}+x^\beta \Delta_{M}\ \text{with } C_\beta=\frac{\beta n}{4}\left(1+\frac{\beta n}{4}\right) . 
	\end{equation}

We are now ready to introduce the trace asymptotics for the Laplace--Beltrami operator. Our method relies on reducing our problem to the case of a simpler one-dimensional operator that was also used in \cite{de_verdiere_weyl_2024}.
 
	\begin{lemma}\label{lem:trace}
		Let $V$ be bounded with $V\in C(X_\infty)$, then with $V_\lambda$ defined in \eqref{eqn:Vlambda}, it follows that  
		\begin{align}\label{eq:lecd}
			\lim_{\lambda\rightarrow\infty}\lambda^{-1-\frac{d}{2}}\Tr\left(\Delta_X+V_\lambda-\lambda\right)_{-}=L^{\mathrm{cl}}_{0,n}\int_M \int_{0}^\infty \Tr_{\R_+}(P_{s^{2/n}}+V-1)_{-}\dd s\dd v_G,
		\end{align}
		where $L^{\mathrm{cl}}_{0,n}\coloneqq (4\pi)^{-n/2}\Gamma(1+n/2)^{-1}$ is the semiclassical constant and $P_\mu$ for $\mu>0$ is defined below in \eqref{eqn:Pmu}.
	\end{lemma}
	\begin{proof}
 
After having performed the change of coordinate, functions and measure as described above, we apply Dirichlet--Neumann bracketing and impose Dirichlet or Neumann conditions along $x=\varepsilon$ so that
	\begin{align}\label{eq:dnbracketing}
		\Delta_{g_{\varepsilon}}\vert^N_{X_\varepsilon}\oplus \Delta_{g_{\varepsilon}}\vert^N_{X\backslash X_\varepsilon}\leq \Delta_{g_\varepsilon}\leq \Delta_{g_{\varepsilon}}\vert^D_{X_\varepsilon}\oplus \Delta_{g_{\varepsilon}}\vert^D_{X\backslash X_\varepsilon}.
	\end{align}
The parts corresponding to $X\setminus X_\epsilon$ are of subleading order by the Weyl law on smooth Riemannian manifolds with boundary, see e.g.~\cite{BGM, Dav},  and  $n+1<d$:
\begin{equation}\label{eq:subleadingint}
    \lim_{\lambda\rightarrow\infty}\lambda^{-1-\frac{d}{2}}\Tr\left(-\Delta_{g_\varepsilon}\vert^{D/N}_{X\setminus X_\varepsilon}+V_\lambda-\lambda\right)_{-}=0.
\end{equation}
 
For the parts corresponding to $X_\epsilon=[0,\epsilon)\times M$, we write the function $V_\lambda$ in local coordinates $V_\lambda(x,y)=\lambda V(x\sqrt{\lambda},y)$. In the following, we will only treat the case where $V$ is independent of $y$, but we allow for $M$ to have a boundary with Dirichlet or Neumann boundary conditions on $[0,\epsilon)\times \partial M$. The general case can then be deduced from this case, using the continuity of $V$, by decomposing $M$ into smaller subdomains and choosing $V$ constant there and Dirichlet--Neumann bracketing.

\medskip

Let us define
\begin{equation}\label{eqn:Pmu}
    P_{\mu}\coloneqq -\partial_x^2+\frac{C_\beta}{x^2}+\mu x^\beta
\end{equation}
which we always consider with Dirichlet conditions at zero. Denote the eigenvalues of $\Delta_M$ by $(\mu_j)_{j=0}^\infty$, then by using separation of variables, the operators decompose in the eigenbasis of $\Delta_M$ and we have
	\begin{align*}
		\Tr\left(-\Delta_{g_\varepsilon}\vert^{D/N}_{ X_\varepsilon}+V_\lambda-\lambda\right)_{-}&= \sum_{j=0}^\infty\Tr_{[0,\epsilon)}^{D/N}\left( P_{\mu_j}+V_\lambda(x)-\lambda\right)_{-},
	\end{align*}
where we note that the terms are non-zero for only finitely many $j$. 

For any fixed $j$ each of the terms in the sum above has sub-leading contribution. Moreover, from the standard form of Weyl's law the eigenvalues of $\Delta_M$ satisfy
	\begin{equation*}
		\mu_j=c_{M}j^{2/n}(1+o(1)) \text{ as }j\rightarrow \infty,
	\end{equation*}
	with $c_{M}=(L^{\mathrm{cl}}_{0,n}v_G(M))^{-2/n}$. Namely, for any $\delta>0$ we can find $K\ge2$ sufficiently large so that 
	\begin{equation*}
		(1+\delta)^{-1}c_{M} j^{2/n}\leq \mu_j\leq (1+\delta)c_{M} j^{2/n}
	\end{equation*}
	for all $j\geq K$. In the trace estimate \eqref{eqn:liminfsup} we can henceforth assume that the eigenvalues $\mu_j=c_M j^{n/2}$ by removing $K$ terms in the sum and taking $\delta$ small. 

 \medskip

 Hence,
 \begin{equation}\label{eqn:decmptrlim}
    \lim_{\lambda\rightarrow\infty}\lambda^{-1-\frac{d}{2}}\Tr\left(-\Delta_{g_\varepsilon}\vert^{D/N}_{ X_\varepsilon}+V_\lambda-\lambda\right)_{-}=\lim_{\lambda\rightarrow\infty}\lambda^{-1-\frac{d}{2}} \sum_{j=K}^\infty\Tr_{[0,\epsilon]}^{D/N}\left( P_{c_Mj^{2/n}}+V_\lambda(x)-\lambda\right)_{-}.
\end{equation}

First note that the minimum of the function $x\mapsto \frac{C_\beta}{x^2}+\mu x^\beta$ is attained at a value that is proportional to $\mu^{2/(2+\beta)}$, so there exists a constant $C>0$ such that  for all $j\ge C\lambda^{d/2}$, $\frac{C_\beta}{x^2}+c_Mj^{2/n} x^\beta>\lambda$, and therefore all terms for $j\ge C\lambda^{d/2}$ in \eqref{eqn:decmptrlim} are zero. 

\medskip

Furthermore, note that for any $L>0$ the part coming from $j\le L\lambda^{n/2}$ is also of subleading order, namely due to Weyl's law
\begin{equation*}
    \Tr_{[0,\epsilon]}^{D/N}\left( P_{c_Mj^{2/n}}+V_\lambda(x)-\lambda\right)_{-}\lesssim \lambda(1+\epsilon\sqrt{\lambda}),
\end{equation*}
therefore 
\begin{equation*}
    \sum_{j=K}^{L\lambda^{n/2}}\Tr_{[0,\epsilon]}^{D/N}\left( P_{c_Mj^{2/n}}+V_\lambda(x)-\lambda\right)_{-}\lesssim \lambda^{n/2}\lambda(1+\epsilon\sqrt{\lambda})\ll \lambda^{1+\frac{d}{2}}.
\end{equation*}
Thus, 
 \begin{equation}\label{eqn:finsum}
    \begin{split}&\lim_{\lambda\rightarrow\infty}\lambda^{-1-\frac{d}{2}}\Tr\left(-\Delta_{g_\varepsilon}\vert^{D/N}_{ X_\varepsilon}+V_\lambda-\lambda\right)_{-}\\
    &=\lim_{\lambda\rightarrow\infty}\lambda^{-1-\frac{d}{2}} \sum_{j=L\lambda^{n/2}}^{C\lambda^{d/2}}\Tr_{[0,\epsilon]}^{D/N}\left( P_{c_Mj^{2/n}}+V_\lambda(x)-\lambda\right)_{-}.
    \end{split}
\end{equation}

In \cite[Proposition 5]{de_verdiere_weyl_2024} the authors showed that for any $\mu>0$ and $a>0$, the operator $P_\mu$ on $[0,a]$ is  unitarily equivalent to $\mu^{\frac{2}{2+\beta}}P_1$ on $[0,\mu^{1/(2+(\beta)}a]$ via the unitary transformation 
 \begin{equation}\label{eqn:unitranP}
    (Uf)(x)\mapsto \mu^{\frac{1}{2(2+\beta)}}f\left(\mu^{\frac{1}{2+\beta}}x\right).
 \end{equation}
 More generally, including $V_\lambda(x)=\lambda V(\sqrt{\lambda}x)$, we have that for any $\mu>0$, $\lambda>0$ and $a>0$, the operator $P_{\mu}+V_\lambda$ on $[0,a]$ is  unitarily equivalent to $\lambda\left(P_{\mu \lambda^{-1-\beta/2}}+V\right)$ on $[0,a\sqrt\lambda]$. To see this, we consider for any $\phi\in C^\infty_c((0,\epsilon])$ the unitary transformation  $\phi_\lambda(x)=\lambda^{1/4}\phi(x\sqrt{\lambda})$, then 
		\begin{align*}
			\langle (P_\mu+V_\lambda-\lambda)\phi_\lambda,\phi_\lambda\rangle_{L^2[0,\varepsilon]}&=\int_0^\epsilon \abs{\partial_x\phi_\lambda}^2+x^{-2}\abs{\phi_\lambda}^2+(\mu x^\beta+V_\lambda (x)-\lambda)\abs{\phi_\lambda}^2\dd x\\
			&=\lambda \int_0^{\epsilon\sqrt{\lambda}} \abs{\partial_x\phi}^2+x^{-2}\abs{\phi}^2+(\mu\lambda^{-1-\beta/2} x^\beta+ V(x)-1)\abs{\phi}^2\dd x\\
			&=\lambda \langle (P_{\mu \lambda^{-1-\beta/2}}+V-1)\phi,\phi\rangle_{L^2[0,\varepsilon\sqrt{\lambda}]}.
		\end{align*}
  
    It follows that \eqref{eqn:finsum} changes to 
    \begin{equation*}
    \lim_{\lambda\rightarrow\infty}\lambda^{-1-\frac{d}{2}}\Tr\left(-\Delta_{g_\varepsilon}\vert^{D/N}_{ X_\varepsilon}+V_\lambda-\lambda\right)_{-}=\lim_{\lambda\rightarrow\infty}\lambda^{-\frac{d}{2}} \sum_{j=L\lambda^{n/2}}^{C\lambda^{d/2}}\Tr_{[0,\epsilon\sqrt{\lambda}]}^{D/N}\left( P_{c_M j^{2/n} \lambda^{-1-\beta/2}}+V(x)-1\right)_{-}.
    \end{equation*}
  Now using the fact that each term is non-increasing in $j$ we can estimate the sum from above and below by an integral that, up to a slight change of the limits of integration, is given by
  \begin{equation}\label{eqn:lem4rieint}
    \begin{split}
      &\int_{L\lambda^{n/2}}^{C{\lambda^{d/2}}} \Tr_{[0,\epsilon\sqrt{\lambda}]}^{D/N}\left( P_{c_M s^{2/n} \lambda^{-1-\beta/2}}+V(x)-1\right)_{-}\dd s\\=&L_{0,n}^{\mathrm{cl}}v_G(M)\int_{L\lambda^{n/2-d/2}}^{C} \Tr_{[0,\epsilon\sqrt{\lambda}]}^{D/N}\left( P_{s^{2/n}}+V(x)-1\right)_{-}\dd s.
      \end{split}
  \end{equation}
Using Agmon-type estimates, one can show that for every $s>0$, 
\begin{equation*}
    \lim_{\lambda\rightarrow\infty}\Tr_{[0,\epsilon\sqrt{\lambda}]}^{D/N}\left( P_{s^{2/n}}+V(x)-1\right)_{-}=\Tr_{[0,\infty)}\left( P_{s^{2/n}}+V(x)-1\right)_{-}.
\end{equation*}
  To insert this limit into \eqref{eqn:lem4rieint} we use the dominated convergence theorem. Start by observing 
  \begin{align*}
    \Tr^{D/N}_{[0,\varepsilon\sqrt{\lambda}]}(P_{s^{2/n}}+V(x)-1)_{-}&\leq \Tr^{N}_{[0,\varepsilon\sqrt{\lambda}]}(P_{s^{2/n}}-\norm{V}_\infty-1)_{-}\\
      &\leq  \Tr^{N}_{[0,\min \{\varepsilon\sqrt{\lambda},c s^{-2/n\beta}\}]}(-\partial_x^2-\norm{V}_\infty-1)_{-}
  \end{align*}
   for some $c>0$, where that we use that for all $x\ge c s^{-2/n\beta}$ the potential $C_\beta x^{-2}+s^{2/n}x^\beta-\norm{V}_\infty-1\geq 0$. Note that we can choose $L=L(\eps)>0$ large enough such that $\varepsilon\sqrt{\lambda}\ge c s^{-2/n\beta}$ for all $s\ge L \lambda^{n/2-d/2}$. We obtain for all $s\ge L \lambda^{n/2-d/2}$
  \begin{align*}
       \Tr^{D/N}_{[0,\varepsilon\sqrt{\lambda}]}(P_{s^{2/n}}+V(x)-1)_{-}\leq \Tr^{N}_{[0,c s^{-2/n\beta}]}(-\partial_x^2-\norm{V}_\infty-1)_{-}\lesssim (1 + s^{-2/(n\beta)}),
  \end{align*}
  which is integrable in $s$ from $0$ to $C$ since $\beta>2/n$, and thus acts as a dominating function for \eqref{eqn:lem4rieint}. 

\medskip

  We obtain
\begin{equation*}\label{eq:asymptoticsinterior}
       \lim_{\lambda\rightarrow\infty}\lambda^{-1-\frac{d}{2}}\Tr\left(-\Delta_{g_\varepsilon}\vert^{D/N}_{ X_\varepsilon}+V_\lambda-\lambda\right)_{-}=L_{0,n}^{\mathrm{cl}}v_G(M)\int_{0}^{\infty} \Tr_{[0,\infty)}\left( P_{s^{2/n}}+V(x)-1\right)_{-}\dd s,
  \end{equation*}
  so using this together with \eqref{eq:dnbracketing} and \eqref{eq:subleadingint}, we get
  \begin{equation*}
      \lim_{\lambda\rightarrow\infty}\lambda^{-1-\frac{d}{2}}\Tr\left(\Delta_X+V_\lambda-\lambda\right)_{-}=L_{0,n}^{\mathrm{cl}}v_G(M)\int_{0}^{\infty} \Tr_{[0,\infty)}\left( P_{s^{2/n}}+V(x)-1\right)_{-}\dd s.
  \end{equation*}
  Using the argument described at the beginning of the proof, we can deduce \eqref{eq:lecd} also for continuous functions $V(x,y)$ that depend on $y$. This completes the proof of \eqref{eq:lecd}.
  	\end{proof}

    \section{Proof of Theorem \ref{thm:main_grushin}}\label{sec:pfthm1}
	In this section we use Lemma \ref{lem:trace} to prove Theorem \ref{thm:main_grushin} by carrying through a weak-type convergence argument, see \cite{lieb_thomas-fermi_1977, evans_counting_1996}, using a rescaled Hellmann-Feynman argument. This type of argument was also used in \cite{read_asymptotic_2024}, in which the author studied the low--lying states in the Stark effect. 
    \begin{proof}[Proof of Theorem \ref{thm:main_grushin}]
        Take $\Gamma_\lambda\coloneqq (\Delta_X-\lambda)^{0}_{-}$ and $\widetilde{\rho}_\lambda$ as its density. Note that $\rho_\lambda=N(\lambda)^{-1}\widetilde{\rho}_\lambda$. We start with the lower bound, where 
    	\begin{align*}
    		\int_{X} \widetilde{\rho}_\lambda V_\lambda\dd v_g &=\Tr\left((\Delta_X+V_\lambda-\lambda)\Gamma_\lambda\right)-\Tr\left((\Delta_X-\lambda)\Gamma_\lambda\right)\\
    		&\geq \Tr(\Delta_X-\lambda)_{-}-\Tr(\Delta_X+V_\lambda-\lambda)_{-}
    	\end{align*}
	    follows from the variational principle, given that $0\leq \Gamma_\lambda\leq 1$. Thus, applying the asymptotics from Lemma \ref{lem:trace} 
    	\begin{align*}
    		\liminf_{\lambda\rightarrow\infty}\lambda^{-\frac{d}{2}-1}\int_X \widetilde{\rho}_\lambda V_{\lambda}\dd v_g &\geq 
    		L_{0,n}^{\mathrm{cl}}\int_{M}\int_{0}^\infty\Tr\left(P_{s^{2/n}}-1\right)_{-}-\Tr\left(P_{s^{2/n}}+V-1\right)_{-}\dd s\dd v_G.
    	\end{align*}

        \medskip
    	Substituting $V$ with $\epsilon V$, where $\varepsilon>0$, in the above shows that
    	\begin{equation}\label{eqn:densityepsilon}
    		\begin{split}&\liminf_{\lambda\rightarrow\infty}\lambda^{-\frac{d}{2}-1}\int_X \widetilde{\rho}_\lambda V_{\lambda}\dd v_g\\ &\geq
    		\varepsilon^{-1} L_{0,n}^{\mathrm{cl}}\int_M \int_{0}^\infty\Tr\left(P_{s^{2/n}}-1\right)_{-}-\Tr\left(P_{s^{2/n}}+\varepsilon V-1\right)_{-}\dd s\dd v_G.
        \end{split}
    	\end{equation}

      For ease of notation, we define 
    	\begin{align*}
    		\Lambda(s,y;V,\varepsilon)\coloneqq \varepsilon^{-1}\left[\Tr\left(P_{s^{2/n}}-1\right)_{-}-\Tr\left(P_{s^{2/n}}+\varepsilon V(x,y)-1\right)_{-}\right].
    	\end{align*}
    	Applying perturbation theory, see for example \cite[Chapter X.II]{reed_methods_1979}, for each $s$ and $q$ we have that 
    	\begin{align}\label{eqn:perturblim}
    		\lim_{\varepsilon\rightarrow 0}\Lambda(s,y;V,\varepsilon)=\sum_{k=1}^{n_s(1)}\int_{\R_+}\abs{\phi_k(x;s)}^2 V(x,q)\dd x
    	\end{align}
 where we define $n_s(\kappa)$ as the number of eigenvalues of $P_{s^{2/n}}$ less than $\kappa$, and $\phi_k$ denote the corresponding normalised eigenfunctions. 
        \medskip

     Next, we would like to take the limit $\varepsilon\rightarrow 0_+$ in \eqref{eqn:densityepsilon} using the dominated convergence theorem. To this end, let us find a corresponding dominating function.

        Choosing $\varepsilon>0$ small enough so that $\varepsilon<\norm{V}^{-1}_\infty$, then 
    	\begin{align*}
    		\abs{\Lambda(s,y;V,\varepsilon)}&\leq \abs{\Lambda(s,y;-\norm{V}_\infty,\varepsilon)}
    		= \varepsilon^{-1}\sum_{k=1}^{n_s(2)} \left(\lambda_k(P_{s^{2/n}})-\varepsilon\norm{V}_{\infty}-1\right)_{-}-\left(\lambda_k(P_{s^{2/n}})-1\right)_{-}
      \\ &\le n_s(2)\norm{V}_\infty.
    	\end{align*}
    	From the relationship \eqref{eqn:unitranP} and the result \cite[Proposition 4]{de_verdiere_weyl_2024} we know the asymptotics as $s\rightarrow 0$ are given by 
    	\begin{equation*}
    		n_s(\kappa)\lesssim \kappa^{1/2+1/\beta}s^{-2/n\beta}(1+o(1))
    	\end{equation*}
    	for any fixed $\kappa>0$. Furthermore, $n_s$ is non-increasing in $s$ and there exists a finite $S>0$ such that $n_s(\kappa)=0$ for all $s>S$. It follows that 
        \begin{align*}
            \int_0^\infty n_s(2)\dd s<\infty
        \end{align*}
        and thus by the dominated convergence theorem we can take the limit $\varepsilon\rightarrow 0_+$ in \eqref{eqn:densityepsilon} and use \eqref{eqn:perturblim} to find that  
    	\begin{align*}
    		\liminf_{\lambda\rightarrow\infty}\lambda^{-1-\frac{d}{2}}\int_{X}\widetilde{\rho}_\lambda V_\lambda\dd v_g&\geq L^{\mathrm{cl}}_{0,n}\int_M \int_{0}^\infty\Tr\left[\left(P_{s^{2/n}}-1\right)_{-}^0V\right]\dd s\dd v_G\\ 
    		&=L^{\mathrm{cl}}_{0,n}\int_M \int_{0}^\infty V(x,y)\int_{0}^\infty\sum_k^{n_s(1)}\abs{\phi_k(x;s)}^2\dd s\dd x\dd v_G .
    	\end{align*}

    	Taking $-V$ in place of $V$ in the argument above yields the same quantity as an upper bound
    	for the $\limsup$. Therefore, we have shown that for any bounded $V\in C([0,\infty)\times M)$ 
    	\begin{align}\label{eq:Bbeforedef}
\lim_{\lambda\rightarrow\infty}\lambda^{-1}\int_{X}\rho_\lambda V_\lambda\dd v_g
    				&=L^{\mathrm{cl}}_{0,n} C_{\alpha,n}^{-1}\int_M \int_{0}^\infty V(x,y)\left(\int_{0}^\infty\sum_k^{n_s(1)}\abs{\phi_k(x;s)}^2\dd s\right)\dd x\frac{\dd v_G}{v_G(M)} ,
    	\end{align}
        where $C_{\alpha,n}$ is the constant appearing in \eqref{eqn:VDdHTresult}.

        Now, we can choose
        \begin{equation}\label{eq:Bdef}
            B(x):=L^{\mathrm{cl}}_{0,n} C_{\alpha,n}^{-1}\int_{0}^\infty\sum_k^{n_s(1)}\abs{\phi_k(x;s)}^2\dd s.
        \end{equation}
It remains to show that $\int_0^\infty B=1$. To this end, note that taking $V\equiv 1$, 
        from \cite[Theorem 2]{de_verdiere_weyl_2024}, we have that
        \begin{align*}
        \lim_{\lambda\rightarrow\infty}\int_{[0,1)\times M}\rho_\lambda \dd v_g=\lim_{\lambda\rightarrow\infty}\int_{X}\rho_\lambda\dd v_g=1
        \end{align*}
        which completes the proof using \eqref{eq:Bbeforedef}.
	\end{proof}

\section{Application to the study of acoustic modes in gas giants}\label{sec:application}
In this section, we relate the result from Theorem \ref{thm:main_grushin} to the study of the propagation of sound waves in gas planets, compare with \cite{de_verdiere_weyl_2024}.
\medskip

The study of seismology for gas planets is crucial in revealing insights about their internal structure. For example, astrophysicists use ring seismology in Saturn to measure the eigenfrequencies of sound waves within the planet \cite{saturn,marley1991nonradial}. A particularity of gas planets is that the speed of sound goes to zero near the boundary of the planet. Physicists use a model for this behaviour that includes a parameter $\alpha$ depending on the chemical decomposition of the planet. On a mathematical level, we can study the eigenmodes of soundwaves in the gas planet by considering the eigenfunctions of the Laplace--Beltrami operator with a singular metric that depends on this parameter $\alpha$ \cite{de_hoop_geometric_2024}. 

\medskip

The study of this model has been initiated in \cite{de_hoop_geometric_2024} and then continued in \cite{de_verdiere_weyl_2024}. In \cite{de_verdiere_weyl_2024}, the authors proved Weyl asymptotics for all physically relevant values of $\alpha$ and they also showed that the majority of eigenfunctions concentrate at the boundary of the planet for critical and supercritical values of $\alpha$. As an application of Theorem \ref{thm:main_grushin}, we continue this work and show the rate at which the eigenfunctions concentrate near the boundary for supercritical values of $\alpha$.
 
 \medskip
 
	Again, let $X$ be a $n+1$ dimensional, smooth and compact manifold with boundary, which represents the gas planet. We consider on $X$ a singular Riemannian metric $g$, such that near $\partial X$,
 \begin{equation*}
     g = \bar{g}/u^\alpha ,
 \end{equation*}
 where $\bar{g}$ is a smooth non-degenerate Riemannian metric up to the boundary, $0<\alpha<2$ and $u$ is a transverse coordinate to the boundary. We can identify the part of $X$ near the boundary with $[0,1)\times M$, where $M$ is an $n$-dimensional Riemannian manifold and $\{0\}\times M$ is identified with $\partial X$. Choosing this transverse coordinate appropriately, we can write the Riemannian metric $g$ using the normal form \cite{de_hoop_geometric_2024, Graham}
	\begin{equation}\label{eqn:Xsingmetric}
		g=u^{-\alpha}(du^2+g_0(u)),
	\end{equation}
 where $g_0(u)$ is a family of smooth Riemannian metrics on $M$ depending continuously on the variable $u$. 

  \begin{figure}[h]
		\centering
		\begin{tikzpicture}[scale=3, every node/.style={font=\small}]
			\def\R{1}         
			\def\A{0}      
			\def\n{20.5}         
			\def\d{0.3}       
			
			\draw [thick, domain=0:360, samples=200, smooth, variable=\t]
			plot ({ (\R + \A * sin(\n * \t)) * cos(\t) }, { (\R + \A * sin(\n * \t)) * sin(\t) });
			
			\node at (0, 0) {$X$};
			
			\draw [thick, dashed, domain=0:360, samples=200, smooth, variable=\t]
			plot ({ (\R - \d + \A * sin(\n * \t)) * cos(\t) }, { (\R - \d + \A * sin(\n * \t)) * sin(\t) });
			
			\fill [gray!15, even odd rule]
			plot [domain=0:360, samples=200, smooth, variable=\t]
			({ (\R + \A * sin(\n * \t)) * cos(\t) }, { (\R + \A * sin(\n * \t)) * sin(\t) })
			--
			plot [domain=360:0, samples=200, smooth, variable=\t]
			({ (\R - \d + \A * sin(\n * \t)) * cos(\t) }, { (\R - \d + \A * sin(\n * \t)) * sin(\t) })
			-- cycle;
			
			\node at ({ (\R + \A) * cos(45)+ 0.1}, { (\R + \A) * sin(45) + 0.1 }) {$\partial X$};
			
			\draw [->] (-1,0) -- (-0.7,0) node[midway, above] {$u$};
			
			\node at (0,-0.85) {$[0,1)\times M$};
		\end{tikzpicture}
		\caption{Illustrative $X$, grey area represents $[0,1)\times M$}
		\label{fig:enter-label}
	\end{figure}
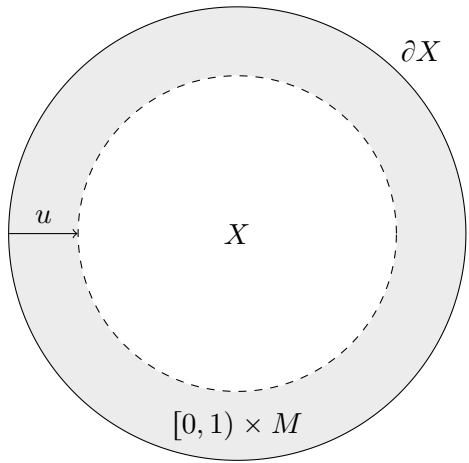

 \medskip

   We start by performing a change of variables on $[0,1)\times M$. Taking $x=(1-\frac{\alpha}{2})^{-1}u^{1-\frac{\alpha}{2}}$, we can write \eqref{eqn:Xsingmetric} as
	\begin{align*}
		g=dx^2+x^{-\beta}g_1(x)\text{ with }\beta=\frac{2\alpha}{2-\alpha},
	\end{align*}
	for some family of smooth Riemannian metrics $g_1 (x)$ on $M$ that depend continuously on $x$. This recovers the Grushin model considered in this paper, compare with \eqref{eq:Grushin_metric}.

\medskip

Using Theorem \ref{thm:main_grushin}, we can deduce that the acoustic modes on gas planets accumulate at a rate of $\lambda^{-1/(2-\alpha)}$. 
    \begin{theorem}\label{thm:main}
		Let $\rho_\lambda$ be the density of eigenfunctions corresponding to eigenvalues less than $\lambda$ of the Laplace--Beltrami operator corresponding to the metric \eqref{eqn:Xsingmetric}, compare with \eqref{eqn:spectralfnc}, and let $\alpha\in (2/(n+1),2)$. Then there exists a function $A\in L^1(0,\infty)$ depending on $\alpha$ with $\norm{A}_1=1$ that is defined in \eqref{eq:Bdef} and \eqref{eq:Adef} below
  such that
		\begin{align*}
			\lim_{\lambda\rightarrow\infty} \int_{[0,1)\times M}\rho_\lambda(u,y) f(\lambda^{1/(2-\alpha)}u,y)\dd v_{g}=\int_{[0,\infty)\times M} A (u)f(u,y)\dd u\frac{\dd v_G(y)}{v_G(M)}
		\end{align*}
		holds for any bounded $f\in C([0,\infty)\times M)$.
	\end{theorem}

    \begin{proof}[Proof of Theorem \ref{thm:main}]
	We obtain the form of the result in Theorem \ref{thm:main} by changing back into the true transversal coordinates $u$. 
 Define for any $v\in [0,\infty)$, $y\in M$
 \begin{equation}\label{eq:fV}
     f(v,y):=V\left(\left(1-\frac{\alpha}{2}\right)^{-1}v^{1-\frac{\alpha}{2}},y\right).
 \end{equation}
 Note that using $x=(1-\alpha/2)^{-1}u^{1-\frac{\alpha}{2}}$, we get 
	\begin{align*}
		V_\lambda(x,y)=\lambda V\left(\left(1-\frac{\alpha}{2}\right)^{-1}(u\lambda^{1/(2-\alpha)})^{1-\frac{\alpha}{2}},y\right)= \lambda f(u\lambda^{1/(2-\alpha)},y).
	\end{align*}
	Taking a change of variables and translating this to the volume form we see that 
	\begin{equation*}
	  \lim_{\lambda\rightarrow\infty}\int_{X_1}\rho_\lambda(u,y)f(u\lambda^{1/(2-\alpha)},y)\dd v_g=   \lim_{\lambda\rightarrow\infty}\lambda^{-1}\int_{X}\rho_\lambda(u,y)V_\lambda\dd v_g
	\end{equation*}
	Using Theorem \ref{thm:main_grushin}, we obtain
 \begin{equation*}
       \lim_{\lambda\rightarrow\infty}\int_{X_1}\rho_\lambda(u,y)f(u\lambda^{1/(2-\alpha)},q)\dd v_g= \int_{(0,\infty)\times M} V(x,y)B(x)\dd x\frac{\dd v_{G}(y)}{v_G(M)}.
 \end{equation*}
Now defining 
\begin{equation}\label{eq:Adef}
    A(u)= B((1-\alpha/2)^{-1}u^{1-\alpha/2})u^{-\alpha/2},
\end{equation}
where $B$ is given \eqref{eq:Bdef}, we have that
\begin{equation*}
    \lim_{\lambda\rightarrow\infty}\int_{X_1}\rho_\lambda(u,y)f(u\lambda^{1/(2-\alpha)},q)\dd v_g= \int_{(0,\infty)\times M} f(u,y)A(u)\dd u\frac{\dd v_{G}(y)}{v_G(M)}. \qedhere
\end{equation*}
 \end{proof}

   Another way of writing the statement of Theorem \ref{thm:main} is to move the scaling onto the spectral function.    \begin{corollary}\label{rem:rescale}
	Let $\rho_\lambda$ be as in Theorem \ref{thm:main}, then 
  \begin{equation}\label{eq:distrrem}
			\lim_{\lambda\rightarrow\infty}\lambda^{-\frac{1}{2}+\frac{\alpha d}{4}}\int_{[0,\infty)\times M}\rho_\lambda(\lambda^{-1/(2-\alpha)}u,y)f(u,y)\dd v_{\widetilde{g}}=\int_{[0,\infty)\times M} \PR (u)f(u,y)\dd u\frac{ \dd v_G}{v_G(M)}.
		\end{equation}
        for any $f\in C_c([0,\infty)\times M)$, where $\tilde g=u^{-\alpha}(\dd u^2+g_0(0))$.
 \end{corollary}
 \begin{proof}
    Using the change of function in \eqref{eq:fV} and moving this onto the spectral function we see that 
	\begin{equation*}
        \begin{split}
	    	 &\lim_{\lambda\rightarrow\infty}\int_{X_1}\rho_\lambda(u,q)f(u\lambda^{1/(2-\alpha)},y)\dd v_g(y)\\
		 =	&\lim_{\lambda\rightarrow\infty}\lambda^{-\frac{1}{2}+\frac{d\alpha}{4}}\int_{[0,\lambda^{-1/(2-\alpha)}]\times M}\rho_\lambda(\lambda^{-1/(2-\alpha)}u,q)f(u,y)\dd v_{g_\lambda}(y),
        \end{split}
	\end{equation*}
	where $g_\lambda=u^{-\alpha}(\dd u^2+g_0(u\lambda^{-1/(2-\alpha)}))$. Then for continuous $f$ that are compactly supported in $u$, given that $g_0(u)$ is continuous in $u$, we can replace the Riemannian measure $ v_{g_\lambda}$ by $ v_{\tilde g}$, where $\tilde g=u^{-\alpha}(\dd u^2+g_0(0))$ and obtain \eqref{eq:distrrem}.
 \end{proof}

	\subsection*{Acknowledgements} 
	C.D.~expresses her deepest gratitude to Phan Thành Nam and
Laure Saint-Raymond for their continued support. She is grateful to Yves Colin de Verdière and Emmanuel Trélat for introducing her to the Grushin model and for very helpful discussions and advice.
Parts of this work were done while C.D.~stayed at Institut des Hautes Etudes Scientifiques and she would like to thank the institute for the hospitality. She acknowledges the support by the European
Research Council via ERC CoG RAMBAS, Project No. 101044249.
	L.R.~was funded by the Deutsche Forschungsgemeinschaft (DFG) project TRR 352 – Project-ID 470903074.


\begin{thebibliography}{10}

     \bibitem{BGM} 
    M. Berger, P. Gauduchon, and E. Mazet,
    \textit{Le spectre d'une vari\'et\'e Riemannienne},
    Springer Lecture Notes in Maths {\bf 194} (1971). 
    
    \bibitem{BL}
    U. Boscain,  C. Laurent,   \textit{The Laplace-Beltrami operator in almost-Riemannian geometry}, Ann. Inst. Fourier {\bf 63}(5), 1739--1770 (2013).
    
     \bibitem{Boscain}
    U. Boscain, D. Prandi, and M. Seri,
     \textit{Spectral analysis and the {Aharonov}-{Bohm} effect on certain almost-{Riemannian} manifolds},
    Comm. Partial Differential Eq. {\bf 41}(1), 32--50 (2016).
    
    \bibitem{Che}
    J. Cheeger,
    \textit{The spectral geometry of singular Riemannian spaces}, J. Diﬀerential Geom. {\bf 18}(4), 575--657 (1983). 
\bibitem{chitour}
    Y. Chitour, D. Prandi, and L. Rizzi, \textit{Weyl's law for singular Riemannian manifolds}, J. Math. Pures Appl. {\bf 181}, 113-151 (2024). 
    
    \bibitem{de_verdiere_weyl_2024}
        Y.~Colin de~Verdière, C.~Dietze, M.~V. de~Hoop, and E.~Trélat,
        \newblock \textit{Weyl formulae for some singular metrics with application to acoustic
          modes in gas giants}, arXiv:2406.19734 (2024).

        \bibitem{CHT1}
        Y. Colin de Verdi\`ere, L. Hillairet, and E. Tr\'elat,
        \textit{Spectral asymptotics for sub-Riemannian Laplacians, I: Quantum ergodicity and quantum limits in the 3-dimensional contact case},
        Duke Math. J. {\bf 167}(1), 109--174 (2018).

        \bibitem{CHT2} 
        Y. Colin de~Verdi\`ere, L. Hillairet, and E. Tr\'elat, 
        \textit{Small-time asymptotics of hypoelliptic heat kernels near the diagonal, nilpotentization and related results},
        Ann. H. Lebesgue {\bf 4}, 897--971 (2021).

        
        \bibitem{de_verdiere_sub_2022}
        Y.~Colin de~Verdi\`ere, L.~Hillairet, and E.~Tr\'elat, \textit{Spectral asymptotics for sub-Riemannian Laplacian}, arXiv:2212.02920 (2024).
 
        
         \bibitem{Dav} 
     E. B. Davies,
    \textit{ Heat kernels and spectral theory},
    Cambridge Tracts in Maths (1989). 
    \bibitem{de_hoop_geometric_2024}
        M.~V. de~Hoop, J.~Ilmavirta, A.~Kykk\"anen, and R.~Mazzeo,
        \textit{Geometric inverse problems on gas giants}, arXiv:2403.05475 (2024).

        \bibitem{saturn}
        J.~W. Dewberry, C.~R Mankovich, J.~Fuller, D.~Lai, and W.~Xu,
        \textit{Constraining Saturn’s interior with ring seismology: effects of differential rotation and stable stratification}, The Planetary Science Journal, \textbf{2}(5), (2021). 
            
        \bibitem{evans_counting_1996}
        W.~D. Evans, R.~T. Lewis, H.~Siedentop, and J.~P. Solovej, \textit{Counting eigenvalues using coherent states with an application to Dirac and {Schrödinger} operators in the semi-classical limit}, Ark. Mat. \textbf{34}(2), 265--283 (1996). 

        \bibitem{frank_weyls_2023}
        R.~L. Frank, \textit{Weyl’s {Law} under {Minimal} {Assumptions}}, In M.~Brown et al. (eds.): From 
          {Complex} {Analysis} to {Operator} {Theory}: {A} {Panorama}, pp. 549--572, Springer (2023).
          
        \bibitem{Graham}
        C. R. Graham, J. M. Lee,
        \textit{Einstein metrics with prescribed conformal infinity on the ball},
        Adv. Math. {\bf 87}(2), 186--225 (1991).
        
        \bibitem{lieb_thomas-fermi_1977}
        E.~H. Lieb, B.~Simon, \textit{The {Thomas–Fermi} theory of atoms, molecules and solids}, Adv. in Math. \textbf{23}, 22–-116 (1977).

        \bibitem{marley1991nonradial}
        M.~S.~Marley, \textit{Nonradial oscillations of Saturn}, Icarus \textbf{94}(2), 420--435 (1991).
        
        \bibitem{read_asymptotic_2024}
        L.~Read.
        \textit{On the asymptotic number of low-lying states in the two-dimensional
          confined {Stark} effect}, arXiv:2404.14363 (2024).
        
        \bibitem{reed_methods_1979}
        M.~Reed, B.~Simon, \textit{Methods of Modern Mathematical Physics IV: Analysis of Operators}, Academic Press (1978).
\end{thebibliography}
\end{document}